\documentclass[a4paper,11pt]{article}
\title{Metric Diophantine Approximation on Fractals}
\author{James Wyatt}
\usepackage[numbers]{natbib}
\usepackage{cite}

\usepackage{url}% allows use of the \url{...} command to typeset a url properly

\usepackage{graphicx}
\usepackage{framed}
\usepackage{amsmath}
\usepackage{amssymb}
\usepackage{amsfonts}
\usepackage{cite}
\usepackage{mathrsfs}
\usepackage{array}

\usepackage{amsthm} %  package used to make the theorem environments work.

\usepackage{centernot} 	% used for \centernot command which negates a symbol
\usepackage{xfrac} 		% used for \sfrac
\usepackage{array} 		% used to give more options when constructing arrays
\usepackage{siunitx} 	% used for SI units
\usepackage[text]{esdiff} % used to make derivatives in Leibniz notation, the "text" option here makes the derivatives look better when used in-line.
\usepackage{bm}			% used for \bm{...} commands to make bold mathematical symbols
\usepackage{gensymb}	% used for \degree command
\usepackage{eurosym} 	% used for euro symbol
\usepackage{stmaryrd} 	% used for \lightning symbol
\usepackage{tikz}		% used for drawings (I used it to make contradiction symbol)

\theoremstyle{plain} % this sets the style for all new environments created using \newtheorem to have the "theorem" style, which as a bold title, italic text and vertical space above and below it. 
\newtheorem{thm}{Theorem}[section] 
\newtheorem{lem}[thm]{Lemma}

\newtheorem*{Khintchine}{Khintchine's Theorem} % the star makes an unnumbered theorem#
\newtheorem*{Main}{Main result}

\theoremstyle{definition} % this sets the style for all new environments created using \newtheorem to have the "definition" style, which as a bold title, upright text and vertical space above and below it.
\newtheorem{defn}{Definition}[section]

\theoremstyle{remark} % this sets the style for all new environments created using \newtheorem to have the "remark" style, which as an italic non-bold title, upright text and no extra vertical space above and below it.

\begin{document} 
\maketitle
\begin{abstract}
Inspired by a problem proposed by Mahler \citep{KUR}, we will address the following related question, `How well can irrationals in a missing digit set be approximated by rationals with polynomial denominators?' and prove some related results. To achieve this, we will be closely looking at Khintchine's theorem, particularly the convergence case and aim to prove a Khintchine-like convergence theorem for missing digit sets with large bases and rationals with polynomial denominators.
\end{abstract} 
\section{Introduction} 
The result we will prove in this paper is stated as follows.

We define $||x||$ as the distance between $x$ and the nearest integer.

For a polynomial $P$ and missing digit fractal $K$ on $[0,1]$ let,
\[
W_{P(q)}(\psi)=\{x\in K^{d} : \text{max}\{||P(q)x_{i}||\}<\psi(q)~\text{for infinitely many }q\in\mathbb{N}\}.
\]
\begin{Main}
Let $\psi$ be a monotonically decreasing function. For a missing digit measure $m$ on $K^{d} \subset [0,1]^d$, $d \in \mathbb{N}$, with $K$ having a large base and only one missing digit,
\[
m(W_{P(q)}(\psi))=0,~\text{if}~\sum^{\infty}_{q=1}\psi(q)<\infty.
\]

\end{Main}

In 1984, Mahler presented eight problems in transcendental number theory and Diophantine approximation \citep{KUR}. One of these problems posed the questions, 
\begin{itemize}
\item How well can irrationals in the middle third Cantor set be approximated by rationals in the middle third Cantor set?
\item How well can irrationals in the middle third Cantor set be approximated by rationals not in the middle third Cantor set?
\end{itemize}
The second part of this problem serves as the inspiration for this paper. 
In 2005, Kleinbock-Lindenstrauss-Weiss \citep{DMI} framed the problem to ask whether there exists an analogue to Khintchine's theorem \citep{ALE}, an important theorem in metric Diophantine approximation, for the middle third Cantor set. 

To state Khintchine's theorem, we need a notion of what is means for a point to be approximated. We call a point $x \in \mathbb{R}$ \emph{$\psi$-approximable}, if for infinitely many $q\in \mathbb{N}$,
\[
||qx||<\psi(q).
\]
We let $W(\psi)$ be the set of points that are $\psi$-approximable and are able to state Khintchine's theorem.
\begin{Khintchine}[Khintchine, 1924]
\citep{ALE} Let $\psi:\mathbb{N}\to[0,+\infty)$ be a monotonically decreasing function and let $\lambda$ denote the Lebesgue measure. Then,
\[
\lambda(W(\psi))=
\begin{cases}
0 & \text{if } \sum^{\infty}_{q=1}\psi(q)<\infty,\\
1 & \text{if } \sum^{\infty}_{q=1}\psi(q)=\infty.
\end{cases}
\]
\end{Khintchine}
 
In 2024, B\'enard-He-Zhang established the analogue for Khintchine's theorem on self-similar measures on $\mathbb{R}$ \citep{TIM} and then extended this result to $\mathbb{R}^{d}$ \citep{BEN}. Other recent research surrounding this problem includes \citep{MAN, DAV, HAN, SHR, SAM, CHO}.

All of this research is interested with approximating irrationals by rationals with linear denominators and so we will consider approximating irrationals by rationals with polynomial denominators. The middle third Cantor set can be described as a missing digit set of base $3$, motivating our study of missing digit sets. In this paper, we will show an analogue to the convergence case of Khintchine's theorem for approximating irrationals in a missing digit set with a sufficiently large base and one missing digit on $[0,1]^{d}$ by rationals with polynomial denominators.
\section{Preliminaries}
Firstly, we state the convergence case for an analogue of Khnitchine's theorem for approximating by rationals with polynomial denominators. 
For a polynomial $P$ let
\[
W_{P(q)}(\psi)=\{x\in [0,1] : ||P(q)x||<\psi(q)~\text{for infinitely many }q\in\mathbb{N}\}.
\]
\begin{thm}
Let $\psi:\mathbb{N}\to[0,+\infty)$ be a monotonically decreasing function and let $\lambda$ denote the Lebesgue measure. Then,
\[
\lambda(W_{P(q)}(\psi))=
0 \text{ if } \sum^{\infty}_{q=1}\psi(q)<\infty.
\]
\end{thm}

This can be proved simply using the Borel-Cantelli lemma. 

The fractals we will be looking at in this paper are missing digit sets. 
\begin{defn}
Let $b \in \mathbb{N}$, $b\geq2$ and $D\subset\{0,1,...b-1\}$.
We define a missing digit set $K_{b,D}$ to be the set of real numbers in the interval $[0,1]$ with at least one base expansion with digits only from the set $D$. 
\end{defn}
This means that we are able to write the middle third Cantor set as $K_{3,\{0,2\}}$. 

\begin{defn} For a missing digit set $K_{b,D}$, let $D_{i}$, $i\geq1$, be a collection of independently and identically distributed random variables with elements in $D$. We define the \emph{missing digit measure} $\mu_{b,D}$ supported on $K_{b,D}$ to be the Borel probability measure of the distribution of the random number, 
\[
\sum_{i=1}^{\infty}D_{i}b^{-i}.
\]
\end{defn}
Continuing forward in this paper we will use Vinogradov notation, primarily $\ll$. $f \ll g$ if there exists some constant $C>0$ such that $|f|\leq C|g|$ pointwise.

Introduced by Yu in \citep{HAN}, the Fourier $l^{t}$-dimension will be an essential part of the main proofs of this paper.
Let $\nu$ be a Borel probability measure on $[0,1]$. For $\xi \in \mathbb{Z}$, we define the $\xi^{\text{th}}$ Fourier coefficient of $\nu$ as,
\[
\hat{\nu}(\xi)=\int_{0}^{1}e^{-2\pi i \xi x}d\nu(x)
\]
\begin{defn}
We define the \emph{Fourier $l^t$-dimension} as 
\[
\hat{\kappa}_{t}(\nu)=\text{sup}\left\{s\geq0:\sum_{\xi=0}^{Q}|\hat{\nu}(\xi)|^{t}\ll Q^{1-s}\right\}
\]
\end{defn}
\section{Results}

Let $n$ be the order of polynomial $P$ and $K$ a missing digit set.

Define $A(q,\delta)=\{x\in K:||P(q)x||<\delta\}$, inspired by \citep[p.~8]{SAM}.

By choosing the base of $K$ sufficiently large we can ensure that $1-\frac{1}{n}<\hat{\kappa}_{1}(m)$, which we can deduce from \citep[Theorem 2.6]{SAM}.
\begin{lem}
Let $v$ be a number in the range $1-\frac{1}{n}<v<\hat{\kappa}_{1}(m)$. Then
\[
\sum^{2Q}_{q=Q}m(A(q,\delta)) \ll \delta Q
\]
for $Q \in \mathbb{N}$ and $\delta\geq Q^{-(2v-1)/(1-v)}$.
\end{lem}
\begin{proof}
Assume $Q$ is large and let $\varepsilon=\hat{\kappa}(m)-v$.
By applying Parseval's formula as in \citep[Theorem 4.1]{HAN} we have that
\[
m(A(q,\delta)) \ll \delta \sum_{\substack{|\xi|\leq \frac{2P(q)}{\delta}\\ |\xi|=0~mod~P(q)}}|\hat{m}(\xi)|
\]
Let $\tau(\xi)$ denote the number of divisors of $\xi$, we have
\begin{align*}
\sum^{2Q}_{q=Q}m(A(q,\delta)) &\ll \delta \left(Q|\hat{m}(0)|+\sum_{0<|\xi|\leq \frac{2P(2Q)}{\delta}}\tau(\xi)|\hat{m}(\xi)|\right) \\
&\ll \delta \left(Q+\left(\frac{P(Q)}{\delta}\right)^{\frac{\varepsilon}{2}}\sum_{|\xi|\leq \frac{2P(Q)}{\delta}}|\hat{m}(\xi)|\right) \\
&\ll \delta \left(Q+\left(\frac{P(Q)}{\delta}\right)^{\varepsilon+1-\hat{\kappa}(m)}\right) \\
&\ll \delta \left(Q+\left(\frac{P(Q)}{\delta}\right)^{1-v}\right) \\
&\ll \delta \left(Q+\left(\frac{Q^{n}}{\delta}\right)^{1-v}\right)
\end{align*}
Therefore, we have that 
\[
\sum^{2Q}_{q=Q}m(A(q,\delta)) \ll \delta Q
\]
as long as $\delta\geq Q^{-u}$, where
\[
u=\dfrac{1}{1-v}-n=\dfrac{1-n(1-v)}{1-v}
\]
\end{proof}
Now we want to generalise this result to higher dimensions.
For a missing digit set $K$, let us denote its probability measure as $m_{1}$. For $K^{2}=K \times K$ we denote its probability measure as $m_{2}$ and so $K^{d}$ has measure $m_{d}$.
Let $B(q,\delta)= \{ x=(x_{1}, x_{2}, \dotsc , x_{d}) \in K^{d} : \text{max}\{|| P(q)x_{1}||, \dotsc , || P(q)x_{d}|| \} < \delta \}$. 

\begin{lem}
Let $v$ be a number in the range $1-\frac{1}{n}<v<\hat{\kappa}_{1}(m_{1})$. Then
\[
\sum^{2Q}_{q=Q}m_{d}(B(q,\delta)) \ll \delta Q
\]
for $Q \in \mathbb{N}$ and $\delta\geq Q^{-(2v-1)/(1-v)}$.
\end{lem}
\begin{proof}
\[
B(q,\delta) = \bigcap^{d}_{i=1} \{ (x_{1}, x_{2}, \dotsc , x_{d}) \in K^{d} : ||P(q)x_{i}||< \delta  \}
\]
\[
B(q,\delta) \subset \{ (x_{1}, x_{2}, \dotsc , x_{d}) \in K^{d} : ||P(q)x_{1}||< \delta  \}
\]
\[
m_{d} (\{ (x_{1}, x_{2}, \dotsc , x_{d}) \in K^{d} : ||P(q)x_{1}||< \delta  \})=m_{1}(A(q,\delta))
\]
This means that,
\[
m_{d}(B(q,\delta)) \leq m_{1}(A(q,\delta))
\]
So,
\[
\sum^{2Q}_{q=Q}m_{d}(B(q,\delta)) \leq \sum^{2Q}_{q=Q}m_{1}(A(q,\delta))
\]
From Lemma 5.3 we know that,
\[
\sum^{2Q}_{q=Q}m_{1}(A(q,\delta)) \ll \delta Q
\]
Therefore,
\[
\sum^{2Q}_{q=Q}m_{d}(B(q,\delta)) \ll \delta Q
\]
\end{proof}

Finally, we aim to show a convergence Khintchine-like property holds for rationals with polynomial denominators on a missing digit set $K^{d}$.
For a polynomial $P$ let
\[
W_{P(q)}(\psi)=\{x\in K^{d} : \text{max}\{|| P(q)x_{i}|| \} <\psi(q)~\text{for infinitely many }q\in\mathbb{N}\}.
\]

\begin{thm}
Let $\psi$ be a monotonically decreasing function. For a missing digit measure $m_{d}$ on a missing digit set on $[0,1]^{d}$ with a large base and only one missing digit,
\[
m_{d}(W_{P(q)}(\psi))=0,~\text{if}~\sum^{\infty}_{q=1}\psi(q)<\infty
\]

\end{thm}
\begin{proof}
By choosing our base to be sufficiently large, we can ensure $\hat{\kappa}_{1}(m_{1})>1-\frac{1}{n+1}$.

Let $v$ be a number in the range $1-\frac{1}{n+1}<v<\hat{\kappa}_{1}(m)$.

Choose $\varepsilon>0$ sufficiently small such that:
\begin{enumerate}
\item$\varepsilon<\hat{\kappa}_{1}(m_{1})-v$
\item$\varepsilon<\dfrac{(n+1)v-n}{1-v}$
\item$f(Q) \ll  Q^{-(1+\varepsilon)}$.
\end{enumerate}
Since $f$ is convergent
\[
\sum^{2Q}_{q=Q}m_{d}(A(q,f(q))) \leq \sum^{2Q}_{q=Q}m_{d}(A(q,f(Q)) \ll \sum^{2Q}_{q=Q}m_{d}(A(q,Q^{-(1+\varepsilon)}))
\]
From the above lemma we then have that
\[
\sum^{2Q}_{q=Q}m_{d}(A(q,f(q))) \ll Q^{-(1+\varepsilon)}Q = Q^{-\varepsilon}
\]
\begin{align*}
\sum_{q=Q}^{\infty}m_{d}(A(q,f(q)))&\leq \sum^{2Q}_{q=Q}m_{d}(A(q,f(q)))+\sum^{4Q}_{q=2Q}m_{d}(A(q,f(q))) + \dotsm \\
&\ll Q^{-\varepsilon} + \left(2Q\right)^{-\varepsilon} +  \left(4Q\right)^{-\varepsilon} + \dotsm \\
&\ll Q^{\varepsilon}
\end{align*}
as $Q\to \infty$ we have $Q^{-\varepsilon}\to 0$ and so
\[
\sum^{\infty}_{q=Q}m_{d}(A(q,f(q))) \to 0
\]
This implies the convergence Khintchine property.
\end{proof}

\bibliography{URSS_biblio}

@article{KUR,
author={Kurt Mahler},
  title={Some suggestions for further research},
  journal={Bulletin of the Australian Mathematical Society},
  volume={Volume 29},
  pages={pp.101--108},
  year={1984}
}

@article{ALE,
author = {Aleksandr Khintchine},
journal = {Fundamenta Mathematicae},
pages = {9-20},
title = {Über einen Satz der Wahrscheinlichkeitsrechnung},
volume = {6},
year = {1924},
}

@misc{HAN,
      title={Rational points near self-similar sets}, 
      author={Han Yu},
      year={2021},
      eprint={2101.05910},
      archivePrefix={arXiv},
      note = {arXiv:2101.05910 [math.NT]},
}

@misc{SAM,
      title={Counting rationals and diophantine approximation in missing-digit Cantor sets}, 
      author={Sam Chow and Peter Varju and Han Yu},
      year={2024},
      eprint={2402.18395},
      archivePrefix={arXiv},
      note= {arXiv:2402.18395 [math.NT]},
}

@article{DMI, 
title={On Fractal Measures and Diophantine Approximation}, 
volume={10},  
journal={Selecta Mathematica}, 
author={Dmitry Kleinbock and Elon Lindenstrauss and Barack Weiss}, 
year={2005}, 
pages={479-523}}

@misc{TIM,
      title={Khintchine Dichotomy for Self-Similar Measures}, 
      author={Timoth\'ee B\'enard and Weikun He and Han Zhang},
      year={2024},
      eprint={2409.08061},
      archivePrefix={arXiv},
      note={arXiv:2409.08061 [math.DS]},
}

@misc{SHR,
      title={On Fourier Asymptotics and Effective Equidistribution}, 
      author={Shreyasi Datta and Subhajit Jana},
      year={2024},
      eprint={2407.11961},
      archivePrefix={arXiv}, 
      note= {arXiv:2407.11961 [math.NT]},
}

@article{MAN, 
title={Diophantine Approximation on Fractals}, 
volume={21},  
journal={Geometric and Functional Analysis}, 
author={Manfred Einsiedler and Lior Fishman and Uri Shapira}, 
year={2011}, 
pages={14-35}}

@article{DAV, 
title={Random Walks on Homogeneous spaces and Diophantine Approximation on Fractals}, 
volume={216},  
journal={Inventiones Mathmaticae}, 
author={David Simmons and Barack Weis}, 
year={2019}, 
pages={337-394}}

@misc{CHO,
  title        = {Simultaneous and multiplicative Diophantine approximation on missing-digit fractals},
  author       = {Sam Chow and Han Yu},
  year         = {2024},
  eprint       = {2412.12070},
  archivePrefix= {arXiv},
  note         = {arXiv:2412.12070 [math.NT]},
}

@misc{BEN,
  title        = {Khintchine Dichotomy and Schmidt Estimates for Self-Similar Measures on $\mathbb{R}^{d}$},
  author       = {Timoth\'ee B\'enard and Weikun He and Han Zhang},
  year         = {2025},
  eprint       = {2508.09076},
  archivePrefix= {arXiv},
  note         = {arXiv:2508.09076 [math.DS], \url{https://arxiv.org/abs/2508.09076}}
}
\end{document}